\documentclass[11pt,a4paper]{article}
\usepackage{amsmath}
\usepackage{amsthm}
\usepackage{amsfonts}
\usepackage{amssymb}
\usepackage{stmaryrd}
\usepackage{latexsym}

\newtheorem{theorem}{Theorem}
\newtheorem{lemma}{Lemma}

\theoremstyle{definition}
\newtheorem{definition}{Definition}

\begin{document}

 \tolerance2500

\title{\Large{\textbf{On definition of CI-quasigroup}}}
\author{\normalsize {N.N.~Didurik and V.A.~Shcherbacov}
}

 \maketitle

\begin{abstract}
  Groupoid $(Q, \cdot)$ in which equality $(xy)Jx = y$ is true for all $x, y \in Q$, where $J$ is a map of the set $Q$,  is a CI-quasigroup.

\medskip


\medskip

\end{abstract}

\bigskip

\section{Introduction}

Necessary definitions can be found in \cite{RHB, VD, HOP, 2017_Scerb}.

\begin{definition} \label{MAIN_QUAS_DEF} Binary groupoid $(Q, \circ)$ is called a quasigroup if for any ordered pair $(a, b)\in Q^2$
there exist the unique solutions $x, y \in Q$ to the equations $x \circ a = b$ and $a \circ y = b$ \cite{VD}.
\end{definition}

\begin{definition}\label{Loop_Existenc}
A quasigroup $(Q,\cdot)$ with an   element $1 \in Q$, such that $1\cdot x = x\cdot 1 = x$ for all $x\in Q$, is called a {\it loop}.
\end{definition}

We start from  classical definition of Artzy \cite{RA0}.
\begin{definition}
Loop $(Q, \cdot)$ satisfying one of the equivalent identities $x \cdot y  Jx = y$, $xy \cdot Jx = y$, where $J$ is a bijection of the set $Q$ such that $x\cdot Jx=1$,  is called a CI-loop.
\end{definition}
In \cite{RA0} it is proved that $J$ is an automorphism of loop $(Q, \cdot)$.

\begin{definition}
Quasigroup $(Q, \cdot)$ with the identity $xy \cdot Jx = y$, where $J$ is a map of the set $Q$, is called a CI-quasigroup \cite{BEL_TSU}.
\end{definition}
Notice, in this case the map $J$ is a permutation of the set $Q$ \cite{BEL_TSU}. In any CI-quasigroup the permutation $J$ is unique \cite[Lemma 2.25]{2017_Scerb}.

\begin{definition} \label{CI_groupoid}
Groupoid  $(Q, \cdot)$ with the identity
\begin{equation} \label{CI_left}
xy\cdot J_rx=y,
\end{equation}
where $J_r$  is a  map of the set $Q$ into itself, is  called a left CI-groupoid.

Groupoid  $(Q, \cdot)$ with the identity
\begin{equation} \label{CI_right}
J_l x \cdot yx=y,
\end{equation}
where $J_l$  is a map of the set $Q$ into itself, is  called a right CI-groupoid.

Groupoid  $(Q, \cdot)$ with both identities (\ref{CI_left})  and (\ref{CI_right})  is  called a  CI-groupoid.
\end{definition}
Definition \ref{CI_groupoid} is given in  \cite{BEL_TSU}. A groupoid with the equations (\ref{CI_left}) and (\ref{CI_right}) is  called a CI-groupoid in \cite{Izb_Labo}.

Any CI-groupoid is a quasigroup \cite{BEL_TSU}. In  CI-quasigroup the identities (\ref{CI_left}) and (\ref{CI_right}) are equivalent  \cite{BEL_TSU}. From the results of \cite{BEL_TSU} it follows that any left  CI-groupoid is a left quasigroup.

From the results of Keedwell and Shcherbacov (see, for example, \cite[Proposition 3.28]{2017_Scerb}) it follows that  the left CI-groupoid in which the map $J_r$ is  bijective, is a CI-quasigroup.
Any finite left  CI-groupoid is a quasigroup \cite{Izb_Labo}.

\section{Result}

\begin{lemma} \label{Left_Q}
Any left  CI-groupoid is a left quasigroup \cite{BEL_TSU}.
\end{lemma}
\begin{proof}
We prove that in the left CI-groupoid $(Q, \cdot)$ the equation
\begin{equation} \label{Left_quas_Eq}
a\cdot x = b
\end{equation}
 has the unique solution. From the  equation  (\ref{Left_quas_Eq}) we have $ax\cdot J_ra = b\cdot J_r a$, $x= b\cdot J_ra$. If we substitute  last expression in
(\ref{Left_quas_Eq}), then we obtain the following equality:
\begin{equation} \label{Left_quas_Eq_1}
a\cdot bJ_ra = b.
\end{equation}
Uniqueness. Suppose that there exist two solutions of equation  (\ref{Left_quas_Eq}), say, $x_1$ and $x_2$. Then $ax_1=ax_2$, $ax_1\cdot J_r a =ax_2\cdot J_r a$ and from equality (\ref{CI_left}) we obtain that  $x_1=x_2$.

Therefore any left translation $L_x$ of groupoid $(Q, \cdot)$ is a bijective map.
\end{proof}

\begin{lemma} \label{Bijective_CI_Trans}
There exists a bijection between the set $Q$ and the set $\cal R$,  the map $J_r$ is bijective and $J_r Q= Q$.
\end{lemma}
\begin{proof}
We can rewrite the identity (\ref{CI_left})  in the following translation form:
\begin{equation} \label{translation_Equat}
R_{J_rx}L_x =\varepsilon.
\end{equation}
From the equality  (\ref{translation_Equat}) and Lemma \ref{Left_Q} it follows that the map $R_{J_r d}$ is a bijection of the set $Q$ for any fixed element $d\in Q$.

 There exists a bijection between the set $Q$ and the set $\cal{L}$ of all left translations  of groupoid $(Q, \cdot)$. Namely $x\leftrightarrow L_x$, $Q\leftrightarrow \cal L$.

   From the equality (\ref{translation_Equat}) we have that there exists a  bijection between  the set $\cal{L}$ and the set $\cal R$ of all translations (bijections) of the form $R_{J_rx}$, namely, $L_x \leftrightarrow R_{J_rx}$, $\cal L \leftrightarrow \cal R$.
   
Therefore there  exists a bijection between the set $Q$ and the set $\cal R$,  the map $J_r$ is bijective and $J_r Q= Q$.
\end{proof}

\begin{theorem} \label{CI_Group_Is_a_quas}
Any left CI-groupoid $(Q, \cdot)$ is a CI-quasigroup.
\end{theorem}
\begin{proof}
Taking into consideration Lemma \ref{Left_Q} we   must only prove that  in the left CI-groupoid $(Q, \cdot)$ the equation
\begin{equation} \label{Right_quas_Eq}
y\cdot a = b
\end{equation}
 has the unique solution. Using the language of translations we re-write equation (\ref{Right_quas_Eq}) in the following form: $R_a y = b$. By Lemma \ref{Bijective_CI_Trans} the map $R_a$ is a bijection and right translation $R_a$ there exists for any $a \in Q$. Then $y = R^{-1}_a b$.

Therefore any left CI-groupoid $(Q, \cdot)$ is a CI-quasigroup.
\end{proof}

Notice,  Theorem \ref{CI_Group_Is_a_quas} can be proved using Lemmas   \ref{Left_Q},  \ref{Bijective_CI_Trans}  and Proposition 3.28 from  \cite{2017_Scerb}.

\bigskip

\noindent \footnotesize
{Natalia Didurik \hfill Victor Shcherbacov\\
Faculty of Physics and Mathematics \hfill Institute of Mathematics and Computer Science \\
Shevchenko Transnistria State University \hfill  Academy of Sciences of Moldova \\
25 October str., 128, Tiraspol, MD-3300 \hfill 5 Academiei str., Chi\c{s}in\u{a}u  MD-2028 \\
Moldova \hfill  Moldova \\
E-mail: \emph{natnikkr83@mail.ru} \hfill E-mail: \emph{scerb@math.md }}

\end{document}